\documentclass[12pt,leqno,a4paper]{amsart}
\usepackage{amssymb}
\newcommand{\FF}{{\mathbb{F}}}
\newcommand{\NN}{{\mathbb{N}}}
\newcommand{\QQ}{{\mathbb{Q}}}
\newcommand{\ZZ}{{\mathbb{Z}}}

\newcommand{\fS}{{\mathfrak{S}}}

\newcommand{\cK}{{\mathcal{K}}}
\newcommand{\cO}{{\mathcal{O}}}

\newcommand{\Aut}{{\operatorname{Aut}}}
\newcommand{\Gal}{{\operatorname{Gal}}}
\newcommand{\Sp}{{\operatorname{Sp}}}
\newcommand{\sgn}{{\operatorname{sgn}}}
\newcommand{\rnk}{{\operatorname{rnk}}}
\newcommand{\Cl}{{\operatorname{Cl}}}
\newcommand{\id}{{\operatorname{id}}}
\newcommand{\pr}{{\operatorname{pr}}}
\newcommand{\zf}{\cdot10^5}
\newcommand{\zs}{\cdot10^6}

\newcommand\e{{\text{E--}}}

\let\ti=\times
\let\vhi=\varphi

\newtheorem{thm}{Theorem}[section]

\newtheorem{conj}[thm]{Conjecture}

\newtheorem{prop}[thm]{Proposition}

\theoremstyle{definition}

\theoremstyle{remark}

\numberwithin{thm}{section} 
\numberwithin{equation}{section}

\begin{document}

\title{On the distribution of class groups of number fields}

\date{\today}

\author{Gunter Malle}
\address{FB Mathematik, Universit\"at Kaiserslautern,
Postfach 3049, D--67653 Kaisers\-lautern, Germany.}
\makeatletter
\email{malle@mathematik.uni-kl.de}
\makeatother


\begin{abstract}
We propose a modification of the predictions of the Cohen--Lenstra heuristic
for class groups of number fields in the case where roots of unity are
present in the base field. As evidence for this modified formula we provide
a large set of computational data which show close agreement. Furthermore,
our predicted formula agrees with results on class groups of function fields
in positive characteristic for which the base field contains appropriate
roots of unity.

\end{abstract}

\maketitle

\pagestyle{myheadings}
\markboth{Gunter Malle}{On the distribution of class groups of number fields}

\section{Introduction} \label{sec:intro}
The distribution of class groups of number fields remains mysterious. The
Cohen--Lenstra--philosophy, extended by Cohen--Martinet \cite{CM2} gives a
heuristic approach with very precise predictions which is widely expected to
be accurate, but only very few isolated instances have been proved. Recently,
though, we presented computational evidence \cite{Ma} that the
Cohen--Lenstra--heuristic fails for the $p$-part of class groups in the
presence of $p$th roots of unity in the base field. In particular it never
seems to apply for the case $p=2$.  \par
Here, we propose a modified prediction in this case, and
present various computational data in support of this new formula. \par
In Section~\ref{sec:func} we compare our prediction with results on class
groups of function fields, which are related to the distribution of elements
in finite symplectic groups with given eigenspace for the eigenvalue~1. \par

In order to explain our computational results, let's consider
a {\em situation} $\Sigma:=(G,K_0,\sigma)$ consisting of a number
field $K_0$, a transitive permutation group $G$ of degree $n\ge2$, and a
possible signature $\sigma$ of a degree~$n$ extension $K/K_0$ with Galois
group (of the Galois closure) permutation isomorphic to $G$. For such a
situation $\Sigma$, let $\cK(\Sigma)$ denote the set of degree~$n$ extensions
$K/K_0$ of $K_0$ (inside a fixed algebraic closure of $\QQ$) with Galois
group $G$
and signature $\sigma$. We are interested in the structure of the relative
class group $\Cl(K/K_0)$ of $K/K_0$ for $K\in\cK(\Sigma)$ (the kernel in
the class group $\Cl_K$ of the norm map from $K$ to $K_0$). \par
Here we present numerical data for the distribution of $p$-parts of class
groups for the following situations $\Sigma$ and primes $p$:
\begin{itemize}
 \item[(1)] $\Sigma=(C_2, \QQ(\sqrt{-3}), \text{complex})$, $p=3$,
 \item[(2)] $\Sigma=(C_2, \QQ(\mu_5), \text{complex})$, $p=5$,
 \item[(3)] $\Sigma=(\fS_3, \QQ, \text{totally real})$, $p=2$,
 \item[(4)] $\Sigma=(C_3, \QQ, \text{totally real})$, $p=2$,
 \item[(5)] $\Sigma=(C_3, \QQ(\sqrt{-3}), \text{complex})$, $p=2$,
 \item[(6)] $\Sigma=(C_3, \QQ(\sqrt{5}), \text{totally real})$, $p=2$,
 \item[(7)] $\Sigma=(C_3, \QQ(\sqrt{-1}), \text{complex})$, $p=2$,
 \item[(8)] $\Sigma=(D_5, \QQ, \text{complex})$, $p=2$,
 \item[(9)] $\Sigma=(D_5, \QQ, \text{real})$, $p=2$.
\end{itemize}

\section{Class groups in the presence of $p$th roots of unity} \label{sec:conj}

We begin by recalling the setting and the fundamental heuristic assumption in
the paper of Cohen and Martinet \cite{CM2}.

Let $K_0$ be a number field, $K_1/K_0$ a finite extension, $L/K_0$ its Galois
closure with Galois group $G=\Gal(L/K_0)$. (All number fields here are taken
inside a fixed algebraic closure of $\QQ$.) Then $G$ acts on the different
embeddings of $K_1$ into $L$ by the transitive permutation representation on
its subgroup $\Gal(L/K_1)$. The corresponding permutation character $\chi$
contains the trivial character $1_G$ exactly once, and we let
$\chi_1:=\chi-1_G$.
Let $\QQ[G]$ denote the rational group ring of the Galois group $G$. We make
the following two assumptions, which will be satisfied in all examples
considered:
\begin{itemize}
 \item[(1)] $\chi_1$ is the character of an irreducible (but not necessarily
  absolutely irreducible) $\QQ[G]$-module,
 \item[(2)] any absolutely irreducible constituent $\vhi$ of $\chi_1$ has
  Schur index~1, that is, $\vhi$ is the character of a representations of $G$
  over the field of values of $\vhi$.
\end{itemize}
Note that (1) implies in particular that $\Gal(L/K_1)$ is a maximal subgroup
of $G$, or equivalently that the extensions $K_1/K_0$ is simple. We write
$\cO$ for the ring of integers of the field of values of any absolutely
irreducible constituent $\vhi$ of $\chi_1$. (This is an abelian, hence normal
extension of $\QQ$, and thus independent of the choice of constituent $\vhi$
by~(1) above.)
\par
Denote by $E_L$ the group of units of the ring of integers of $L$. Then the
action of $G$ makes $E_L\otimes_\ZZ\QQ$ into a $\QQ[G]$-module, whose character
we denote by $\chi_E$. (It can be computed explicitly in terms of the
signature $\sigma$ of $L/K_0$ by the theorem of Herbrandt, see
\cite[Th.~6.7]{CM2}.) We set
$$u:=\langle\chi_E,\vhi\rangle$$
(see \cite[p.63]{CM2}), the scalar product of the character $\chi_E$ with an
absolutely irreducible constituent $\vhi$ of $\chi_1$. Since $\chi_E$ is
rational, this does not depend on the choice of $\vhi$.
\par
Let's denote by $\cK(\Sigma)$, where $\Sigma=(G,K_0,\sigma)$, the set of number
fields $K/K_0$ with signature $\sigma$ and Galois group of the Galois closure
permutation isomorphic to $G$. Note that both the isomorphism type of the
$\QQ[G]$-module $E_L\otimes\QQ$ and the integer $u$ only depend on the
situation $\Sigma$, not on $K_1$ or $L$. We are interested in the distribution
of relative class groups of fields in $\cK(\Sigma)$. 

By the fundamental assumption of \cite[Hyp.~6.6]{CM2} there should be a notion
of \emph{good primes} for $\Sigma$, including in particular all primes not
dividing $|G|$, and maybe even those not dividing the permutation degree of
$G$, such that whenever $p$ is good for $\Sigma$ and $u\ge1$, then a given
finite $p$-torsion $\cO$-module $H$ should occur as Sylow $p$-subgroup
of a class group $\Cl(K/K_0)$ for $K\in\cK(\Sigma)$ with probability
$$\frac{c}{|H|^u\,|\Aut_\cO(H)|}$$
for some constant $c$ only depending on $p$ and $\Sigma$
(see \cite[Th.~5.6(ii)]{CM2}).

The computational data obtained in \cite{Ma} indicated that this latter
assertion is probably wrong for primes $p$ such that $K_0$ contains $p$th
roots of unity, that is, such primes are not good for $\Sigma$. Based on
further and more extensive computations, instead we propose a modified formula
at least for the case when no $p^2$th roots of unity lie in $K_0$, and
$\cO=\ZZ$:

\begin{conj}  \label{conj:main}
 Assume that $p$ does not divide the permutation degree of $G$ and that $K_0$
 contains the $p$th but not the $p^2$th roots of unity. Then a given finite
 $p$-group $H$ of $p$-rank $r$ occurs as Sylow $p$-subgroup of a relative
 class group $\Cl(K/K_0)$ for $K\in \cK(\Sigma)$ with probability
 $$c\, \frac{\prod_{i=1}^{r+u}(p^i-1)}{p^{r(u+1)}}\cdot
   \frac{1}{|H|^u\,|\Aut(H)|}\,,$$
 where 
 $$c=\frac{1}{\prod_{i=u+1}^\infty(1+p^{-i})}=
     \frac{(p^2)_u(p)_\infty}{(p)_u(p^2)_\infty}$$
 and $u=u(\Sigma)$ is as introduced above.
\end{conj}

Here, for $q,k\in\NN$ we let
$$(q)_k:=\prod_{i=1}^k(1-q^{-i}),\qquad
  (q)_\infty:=\prod_{i=1}^\infty(1-q^{-i}).$$

\begin{prop}  \label{prop:rankmoment}
 Assume that the Sylow $p$-subgroups of class groups $\Cl(K/K_0)$ for
 $K\in \cK(\Sigma)$ are distributed according to Conjecture~\ref{conj:main}.
 Then the probability that $\Cl(K/K_0)$ has $p$-rank equal to $r$ is given by
 $$\pr(\rnk_p(\Cl(K/K_0))=r)=\frac{(p^2)_u(p)_\infty}{(p)_u(p^2)_\infty}\cdot
   \frac{1}{p^{r(r+2u+1)/2}(p)_r}$$
 with $n$th higher moments
 $$\prod_{k=1}^n(1+p^{k-u-1})\qquad\qquad(n\in\NN).$$
\end{prop}

\begin{proof}
This follows easily as in \cite[Lemma~2.1 and~2.2]{Ma}.
\end{proof}

Computationally, only very few cases with $\cO\ne\ZZ$ are in reach. The data
obtained there seem to indicate the following generalization of the above
formula: Assume that $p$ is good for $\Sigma$ and that $K_0$ contains the
$p$th but not the $p^2$th roots of unity. Then a given finite $p$-torsion
$\cO$-module $H$ of $\cO$-rank $r$ should occur as Sylow $p$-subgroup of a
class group with probability
\begin{equation}\label{eq:conj2}
  c\,\frac{d^r\prod_{i=1}^{r+u}(q^i-1)}{q^{r(u+1)}}\cdot
   \frac{1}{|H|^u\,|\Aut_\cO(H)|}\end{equation}
for some constant $c$ only depending on $p$ and $\Sigma$, where
$q:=|\cO/p\cO|$ and $d=(\cO:\ZZ)$.

Computations for cases where the base field contains the $p^2$th roots of unity,
for example on the 2-parts of class groups of cubic extensions of the field
of fourth roots of unity (see Section~\ref{sec:6.4}), or on the 3-parts of class
groups of quadratic extensions of the field of ninth roots of unity, show that
while the distribution of $p$-ranks might still be given as in
Proposition~\ref{prop:rankmoment} the distribution of Sylow $p$-subgroups
seems to be different from the formula given in Conjecture~\ref{conj:main}.
We hope to come back to this question in some future investigation.


In Sections~\ref{sec:S2I3}--\ref{sec:D5}, we consider several instances of
situations $\Sigma$, for which we specialize the conjecture and give supporting
computational data.

\section{Relation with class groups of function fields} \label{sec:func}

First, we compare our new formula to results and heuristics for class groups
of global fields in positive characteristic, that is, function fields over
finite fields $\FF$. Here, the base field contains $p$th roots of unity
if $p$ divides $|\FF|-1$.
Now for a prime power $q$ and for $g,r\ge0$, let
$$\alpha_q(g,r):=\frac{|\{x\in\Sp_{2g}(q)\mid
  \dim(\ker(x-\id))=r\}|}{|\Sp_{2g}(q)|},$$
where $\Sp_{2g}(q)$ denotes the symplectic group of dimension~$2g$ over
$\FF_q$. In \cite[Thm.~3.1]{Ac06} Achter proves that the probability for the
class group of a function field over $\FF$ of genus~$g$ to have $p$-rank~$r$
converges to $\alpha_q(g,r)$ as $|\FF|\rightarrow\infty$ with $p$ dividing
$|\FF|-1$ (see also \cite[Thm.~3.1]{Ac08}). Achter \cite[Lemma 2.4]{Ac06}
shows that $\alpha_q(g,r)$ has a limit for $g\rightarrow\infty$. Here we give
an explicit values for this limit, using the explicit formulas for
$\alpha_q(g,r)$ which were obtained by Rudvalis and Shinoda, see
\cite[Cor.~1]{Fu00}:

\begin{prop}
 For any prime power $q$, and $r\ge0$ we have
 $$\lim_{g\rightarrow\infty}\alpha_q(g,r)
   =\frac{(q)_\infty}{(q^2)_\infty}\cdot\frac{1}{q^{r(r+1)/2}(q)_r}.$$
 Thus, the distribution of elements in $\Sp_{2g}(p)$ according to the
 dimension of their eigenspace for the eigenvalue~1 converges to the
 conjectured distribution of $p$-ranks of class groups in
 Proposition~\ref{prop:rankmoment} for unit rank $u=0$.
\end{prop}

\begin{proof}
First assume that $r=2k$ is even. By \cite[Cor.~1]{Fu00} and using
$$|\Sp_{2k}(q)|=q^{k^2}\prod_{i=1}^k(q^{2i}-1)$$
we have
$$\begin{aligned}
  \alpha_q(g,2k)&=\frac{1}{|\Sp_{2k}(q)|}\sum_{i=0}^{g-k}\frac{(-1)^iq^{i(i+1)}}
  {|\Sp_{2i}(q)|q^{2ik}}\\
  &=\frac{1}{|\Sp_{2k}(q)|}
   \sum_{i=0}^{g-k}\frac{(-1)^iq^{-i^2-2ik}}{(1-q^{-2})\cdots(1-q^{-2i})}\\
  &=\frac{1}{|\Sp_{2k}(q)|}\left(1+
   \sum_{i=1}^{g-k}\frac{(q^2)^{-\binom{i}{2}} (-q)^{i(2k+1)}}
   {(1-q^{-2})\cdots(1-q^{-2i})}\right).
\end{aligned}$$
(Note that the exponent $\binom{i}{2}$ at $q^2$ in the numerator in
loc.~cit.~should correctly read $\binom{i+1}{2}$.)
For $g\rightarrow\infty$, the latter converges to
$$\frac{1}{|\Sp_{2k}(q)|}{\prod_{i=k}^\infty(1-q^{-2i-1})}$$
by \cite[Cor.~2.2]{An}. A trivial rewriting gives the value stated in the
conclusion. \par
For odd $r=2k+1$ we have
$$\alpha_q(g,2k+1)=\frac{1}{q^{2k+1}|\Sp_{2k}(q)|}
  \sum_{i=0}^{g-k-1}\frac{(-1)^iq^{i(i+1)}}{|\Sp_{2i}(q)|q^{2i(k+1)}}$$
by \cite[Cor.~1]{Fu00} (again with the corrected power of $q$ in the numerator),
so by a completely analogous calculation we find
$$\lim_{g\rightarrow\infty}\alpha_q(g,r)=
   \frac{1}{q^{2k+1}|\Sp_{2k}(q)|}\prod_{i=k+1}^\infty(1-q^{-2i-1})$$
from which the claim follows easily.
\end{proof}

It seems tempting to speculate that even the distribution of class groups in
the number field case is as given in \cite{Ac06,Ac08} for the corresponding
function field case in the presence of roots of unity. We have not (yet) been
able to match that with our Conjecture~\ref{conj:main}.

\section{Quadratic extensions and odd primes $p$}\label{sec:S2I3}
We now turn to experimental evidence for Conjecture~\ref{conj:main}.
Our first set of examples concerns the $p$-part of class groups of quadratic
extensions of a number field containing the $p$th roots of unity, where
$p=3$ or $p=5$. Here, in the notation of the previous section $G=Z_2$ is of
order~2, and $\chi_1=\sgn$ is its non-trivial linear character, the sign
character of $\fS_2=Z_2$.
\subsection{Quadratic extensions of $\QQ(\sqrt{-3})$}
The smallest such situation occurs for quadratic extensions of the field
$K_0:=\QQ(\sqrt{-3})$ of third roots of unity. Here, $K_0$ has a unique place
at infinity, and Herbrandt's formula gives $\chi_E=\sgn$, so
$u=1$. The prime $p=3$ is good for this situation, but since the third roots
of unity are present, we expect the Cohen--Lenstra--Martinet heuristic to fail.
In \cite[(3)]{Ma} we already proposed that the distribution of 3-ranks $r$ of
class groups should be given by
\begin{equation}\label{eq:3u=1r}
  \pr(\rnk_3(\Cl_K)=r)=\frac{4}{3}\cdot\frac{(3)_\infty}{(9)_\infty}
  \cdot\frac{1}{3^{r(r+3)/2}(3)_r}\end{equation}
with higher moments
$$\prod_{k=1}^n(1+3^{k-2})$$
(see Proposition~\ref{prop:rankmoment}).
According to Cohen et al. \cite[Cor.~1.3]{CDO} asymptotically the number of
quadratic extensions of $K_0$ of discriminant at most $X$ grows linearly
with $X$, with proportionality factor $0.02613532018111...$. Those extensions
which are Galois over $\QQ$ have density~0,
so generically, the Galois closure over $\QQ$ has dihedral Galois group $D_4$
of order~8. In particular, generically the quadratic extensions come in pairs
with the same Galois closure over $\QQ$. So we expect to find roughly
$0.01306766\,X$ quartic extensions of $\QQ$ with intermediate field
$K_0$ and of discriminant at most $X$. \par
Extending the data presented in \cite[Tab.~9]{Ma} we have compiled lists $S$
consisting of the first $|S|$ quadratic extensions of $K_0$ of discriminant
at least $D$, for various values of $D$. The numbers of fields obtained are
in very close accordance with the asymptotic formula derived above. In
Table~\ref{tab:3-ranks} we give the results of
our computations of 3-ranks for these fields. Visibly, the data fit the
prediction in~\eqref{eq:3u=1r} quite closely. \par

\begin{table}[htbp]
\caption{$C_2$-fields over $\QQ(\sqrt{-3})$: 3-ranks and higher moments}
  \label{tab:3-ranks}
$\begin{array}{cc|rrrr|rrr}
 D& |S|& r=0& 1& 2& 3& n=1& 2& 3\cr
\hline
 \ge10^{16}& 2\zs& .8528& .141& .0058& .71\e4& 1.331& 2.648& 10.55\cr
 \ge10^{20}& 4\zs& .8521& .142& .0059& .68\e4& 1.333& 2.656& 10.43\cr
 \ge10^{24}& 2\zf& .8525& .142& .0057& .80\e4& 1.331& 2.650& 10.43\cr
\hline
\span\text{formula \eqref{eq:3u=1r}}& .8520& .142& .0059& .76\e4& 1.333& 2.667& 10.67\cr
\span\text{CL-prediction}& .8402& .158& .0023& .33\e5& 1.333& 2.444& 6.81\cr
\end{array}$
\end{table}

Conjecture~\ref{conj:main} now predicts more precisely that a 3-group $H$ of
3-rank $r$ occurs as Sylow 3-subgroup of a class group of a quadratic
extension of $\QQ(\sqrt{-3})$ with probability
\begin{equation}\label{eq:3u=1}
  2\cdot\frac{(3)_\infty}{(9)_\infty}\cdot
  \frac{3^{(r^2-r)/2}(3)_{r+1}}{|H|\cdot|\Aut(H)|}
\end{equation}
while the original Cohen--Lenstra--Martinet heuristic \cite{CM1,CM2} predicts
a relative frequency of
$$\frac{(3)_\infty}{(3)_1}\cdot\frac{1}{|H|\cdot|\Aut(H)|}.$$

The following Table~\ref{tab:syl3Z2} contains detailed statistics for the Sylow
3-subgroups for the same sets of data as in Table~\ref{tab:3-ranks} by giving
the quotient of the actual number of fields with given Sylow 3-subgroup
divided by the number expected according to formula~\eqref{eq:3u=1}. 

The last line of Table~\ref{tab:syl3Z2} lists the proportion predicted
by~\eqref{eq:3u=1}.
\vskip 1pc
\goodbreak

\begin{table}[htbp]
\caption{$C_2$-fields over $\QQ(\sqrt{-3})$: Sylow 3-subgroups}
\label{tab:syl3Z2}
$\begin{array}{l|rrrrrrrrr}
 D& 1& 3& 9& 3^2& 27& 9\ti3& 81& 27\ti3& 3^3\cr
\hline
 \ge10^{16}& 1.0009&  .995& .998& .981&  .991&  .944& .985& .954& .873\cr
 \ge10^{20}& 1.0001& 1.000& .997& .989& 1.016& 1.009& .999& .942& .858\cr
 \ge10^{24}& 1.0006& .998& 1.001& .983&  .943&  .913& .924& .720& 1.248\cr
\hline
\text{eq.~\eqref{eq:3u=1}}& .852& .126& .014& .0051& .0016& .75\e3& .17\e3& .83\e4& .64\e4\cr
\end{array}$
\end{table}

The table shows a remarkably good agreement with our prediction.

\subsection{Quadratic extensions of $\QQ(\sqrt{5})$ and of $\QQ(\sqrt{-1})$}
We have computed similar data as above for totally real quadratic extensions
of $\QQ(\sqrt{5})$, resp{}. of quadratic extensions of $\QQ(\sqrt{-1})$.
Here, according to \cite[Cor.~1.3]{CDO} the number of expected fields (over
$\QQ$) of discriminant at most $X$ should grow linearly, with proportionality
factor $0.001852542...$ respectively $0.008144834...$. In these situations,
$K_0$ does not contain the third roots of unity. Our results for $p=3$-parts
of class groups are in close agreement with the Cohen--Lenstra--Martinet
prediction, so we don't show the details.

\subsection{Quadratic extensions of $\QQ(\mu_5)$}
Our final set of examples in this section consists of quadratic extensions of
the field $K_0:=\QQ(\mu_5)$ of fifth roots of unity. Here, we expect the prime
$p=5$ to behave differently. The base field has two places at infinity, so
$\chi_E=1+2\,\sgn$ and $u=2$. According to Conjecture~\ref{conj:main} a 5-group
$H$ of 5-rank $r$ occurs as Sylow 5-subgroup of a class group of a quadratic
extension of $\QQ(\mu_5)$ with probability
\begin{equation}\label{eq:5u=2}
\frac{13}{8}\frac{(5)_\infty}{(25)_\infty}\frac{5^{(r^2-r)/2}(5)_{r+2}}
  {|H|^2\,|\Aut(H)|}.
\end{equation}
Thus, by Proposition~\ref{prop:rankmoment} the distribution of 5-ranks should
be given by
\begin{equation}\label{eq:5u=2r}
  \pr(\rnk_5(\Cl_K)=r)=\frac{156}{125}\cdot\frac{(5)_\infty}{(25)_\infty}\cdot
  \frac{1}{5^{r(r+5)/2}(5)_r}
\end{equation}
with higher moments
$$\prod_{k=1}^n(1+5^{k-3}),\qquad n=1,2,\ldots.$$
We have compiled lists $S$ of the first $10^5$ such extensions of discriminant
$D\ge 10^i$, $i=14,18,22$. Again by \cite[Cor.~1.3]{CDO}, the
number of such fields (over $\QQ$) of discriminant at most $X$ should equal
roughly
$$0.12444267...\cdot10^{-5}\,X,$$
which agrees closely with the numbers
obtained here. Table~\ref{tab:5-ranks} shows the distribution of 5-ranks for
these sets of fields, together with old and new prediction. The predictions
for rank~0 and rank~1 are very close together, but according
to~\eqref{eq:5u=2r}, rank~2 should occur about five times more frequently than
for the original prediction, which fits with the data. 

\begin{table}[htbp]
\caption{$C_2$-fields over $\QQ(\mu_5)$: 5-ranks and higher moments}
  \label{tab:5-ranks}
$\begin{array}{cc|rrr|rrr}
 D& |S|& r=0& 1& 2& n=1& 2& 3\cr
\hline
 \ge10^{14}& 10^5& .99089&  .91\e2& .10\e4& 1.0366& 1.2246& 2.285\cr
 \ge10^{18}& 10^5& .99052&  .95\e2& .10\e4& 1.0381& 1.2335& 2.330\cr
 \ge10^{22}& 10^5& .98987& 1.01\e2& .10\e4& 1.0407& 1.2491& 2.411\cr
\hline
\span\text{formula \eqref{eq:5u=2r}}&  .99008& .99\e2& .16\e4& 1.0400& 1.2480& 2.496\cr
\span\text{CL-prediction}& .99002& 1.00\e2& .33\e5& 1.0400& 1.2416& 2.290\cr
\end{array}$
\end{table}

The amount of data computed in this case is not sufficient to obtain reliable
results on the distribution of Sylow subgroups, so these are not shown.

\section{Non-Galois cubic fields}
A further interesting situation for our conjecture occurs for non-Galois cubic
extensions of $\QQ$ with the prime $p=2$.
\subsection{Totally real non-Galois cubic fields}
The number of totally real $\fS_3$-fields of discriminant at most $X$ is
expected to behave asymptotically as
$$ c_1\,X
  -c_2\frac{1}{\sqrt{3}+1}X^{5/6}+o(X^{1/2})$$
where
$$ c_1=0.06932561438172562...,\qquad c_2=0.403483636663946799...$$
(see Roberts \cite[Conj.~3.1]{Ro}). We have computed the first $10^6$ such
fields of discriminant at least $10^i$, where $11\le i\le17$.
Table~\ref{tab:anzS3} compares the actual number of $\fS_3$-fields of
discriminant~$D$ between $D_1\le D\le D_2$ with the number predicted by the
asymptotic formula.

\begin{table}[htbp]
\caption{Totally real $\fS_3$-fields: asymptotic vs{}. actual numbers }
  \label{tab:anzS3}
$\begin{array}{lc|rrrr|cccc}
 D_1& D_2& |S|& \text{expected}\cr
\hline
 10^{11}& 10^{11}+14816837& 10^6& 1\,000\,421\cr
 10^{12}& 10^{12}+14672596& 10^6&    999\,129\cr
 10^{13}& 10^{13}+14613109& 10^6& 1\,000\,810\cr
 10^{14}& 10^{14}+14544488& 10^6&    999\,997\cr
 10^{15}& 10^{15}+14467409& 10^6&    997\,331\cr
 10^{16}& 10^{16}+14496840& 10^6& 1\,001\,158\cr
 10^{17}& 10^{17}+14464985& 10^6& 1\,000\,181\cr
\cr
\end{array}$
\end{table}

For the totally real case the theorem of Herbrandt gives $u=2$. So
Proposition~\ref{prop:rankmoment} predicts the distribution
\begin{equation}\label{eq:2u=2r}
  \pr(\rnk_2(\Cl_K)=r)=\frac{15}{8}\cdot\frac{(2)_\infty}{(4)_\infty}
  \cdot\frac{1}{2^{r(r+5)/2}(2)_r}
\end{equation}
for the 2-ranks of class groups, with higher moments
$$ \prod_{k=1}^n(1+2^{k-3}),\qquad n=1,2,\ldots$$
(this was already  proposed in our previous paper \cite[(5)]{Ma}).
Computational data for this case which reach considerably farther than those
in \cite[Table 10]{Ma} are displayed in Table~\ref{tab:2-rnkS3r}.
 
\begin{table}[htbp]
\caption{Totally real $\fS_3$-fields: 2-ranks and higher moments }
  \label{tab:2-rnkS3r}
$\begin{array}{lr|rrrr|cccc}
 D& |S|& r=0& 1& 2& 3& n=1& 2& 3& 4\cr
\hline
\ge10^{12}&     10^6& .798& .188& .0135& .354\e3& 1.231& 1.79& 3.35& 8.72\cr
\ge10^{14}&     10^6& .793& .192& .0149& .431\e3& 1.240& 1.83& 3.53& 9.84\cr
\ge10^{16}&     10^6& .789& .195& .0157& .507\e3& 1.246& 1.85& 3.64& 10.47\cr
\ge10^{17}&     10^6& .788& .195& .0158& .538\e3& 1.247& 1.86& 3.68& 10.76\cr
\hline
\span\text{formula \eqref{eq:2u=2r}}& .786& .197& .0164& .585\e3& 1.250& 1.87& 3.75& 11.25\cr
\span\text{CL-prediction}& .770& .220& .0098& .090\e3& 1.250& 1.81& 3.20& 7.18\cr
\cr
\end{array}$
\end{table}

Conjecture~\ref{conj:main} predicts that a 2-group $H$ of 2-rank $r$
occurs as Sylow 2-subgroup of a class group of a totally real non-Galois cubic
number field with probability
\begin{equation}\label{eq:2u=2}
  5\cdot\frac{(2)_\infty}{(4)_\infty}\cdot
  \frac{2^{(r^2-r)/2}(2)_{r+2}}{|H|^2\cdot|\Aut(H)|}.
\end{equation}
As evidence for this we give in Table~\ref{tab:syl2S3} the quotient of the
actual number of fields with given Sylow 2-subgroup divided by the number
expected according to formula~(\ref{eq:2u=2}). In addition, in the last two
lines we print the relative frequency according to~(\ref{eq:2u=2}) and
according to the the original Cohen--Lenstra heuristic.

\begin{table}[htbp]
\caption{Totally real $\fS_3$-fields: Sylow 2-subgroups}  \label{tab:syl2S3}
$\begin{array}{lrrrrrrrrrr}
 D& 1& 2& 4& 2^2& 8& 4\ti2& 2^3& 16& 8\ti2& 4^2\cr  
\hline
 \ge10^{10}& 1.032& .905& .885& .670& .883& .667& .32& .85& .70& .57\cr
 \ge10^{12}& 1.015& .956& .955& .829& .927& .814& .62& .91& .79& .80\cr
 \ge10^{14}& 1.008& .975& .983& .917& .969& .885& .72&1.05& .88& .78\cr
 \ge10^{16}& 1.003& .990&\!\!1.008& .964&1.009& .954& .87& .97& .86& .95\cr 
 \ge10^{17}& 1.002& .993& .997& .958& 1.001& .994& .89& .95& 1.04& .91\cr 
\hline
\text{eq.~\eqref{eq:2u=2}}\!\!\!& .852& .126& .014& .0051& .0016& .75\e3& .17\e3& .8\e4& .6\e4& .2\e4\cr
\text{\cite{CM1}}& .840& .140& .016& .0019& .0017& .29\e3& .19\e3& .3\e4& .3\e5& .2\e5\cr
\end{array}$
\end{table}

The table shows a reasonably good agreement with our prediction.



\section{Cyclic cubic fields} \label{sec:Z3}

Our third set of examples concerns Sylow 2-subgroups of cyclic cubic fields
over various base fields. Here $G=Z_3$ is of order~3, $\cO=\ZZ[\mu_3]$ and
$\chi_1$ is the sum of the two non-rational linear characters of $G$.

\subsection{Cyclic cubic fields over $\QQ$}
The smallest situation, where $K_0=\QQ$, was already considered in
\cite[Sec.~2]{Ma}, where extensive computational results for 2-ranks of
class groups were presented. Here $u=1$, so according to
Proposition~\ref{prop:rankmoment} the 2-ranks of class groups should be
distributed according to
\begin{equation}\label{eq:Z32u=1r}
  \pr(\rnk_2(\Cl_K)=2r)=
    \frac{3}{2}\cdot\frac{(2)_\infty(16)_\infty}{(4)_\infty^2}
    \cdot\frac{1}{2^{r(r+2)}(4)_r}
\end{equation}
(see \cite[(1)]{Ma}), while a given 2-torsion $\cO$-module $H$ of (even) 2-rank
$2r$ should occur with probability
\begin{equation}\label{eq:Z32u=1}
  2\frac{(2)_\infty(16)_\infty}{(4)_\infty^2}\cdot
  \frac{2^{r^2}(4)_{r+1}}{|H|\cdot|\Aut_\cO(H)|}
\end{equation}
as Sylow 2-subgroup of a class group of a cyclic cubic number. As evidence
for this, we list in Table~\ref{tab:syl2Z3} the relative
proportions of certain 2-groups as Sylow 2-subgroups of class groups of
$C_3$-fields. Also, in Table~\ref{tab:syl2Z3prim} we give the corresponding
results for fields of prime conductor. The predicted values for some small
2-groups are given in the last line of Tables~\ref{tab:syl2Z3}
and~\ref{tab:syl2Z3prim} respectively (see also \cite[2(a)]{CM1}). \par

\begin{table}[htbp]
\caption{$C_3$-fields: Sylow 2-subgroups}  \label{tab:syl2Z3}
$\begin{array}{lr|rrrrrrrr}
 D& |S|& 1& 2^2& 4^2& 2^4& 8^2& 4^2\ti2^2& 2^6& 16^2\cr
\hline
 \ge10^{20}& 10^5& 1.003&  .988& 1.010&  .878& .749&  .995&  .578& 1.23\cr
 \ge10^{22}& 10^5& 1.000&  .999&  .994&  .969& .999& 1.288&  .868&  .61\cr
 \ge10^{24}& 10^5& 1.001&  .989& 1.067& 1.063& .749& 1.054& 1.157&  .31\cr
 \ge10^{26}& 10^5& 1.000& 1.002&  .996&  .963& .960& 1.024&  .868& 1.84\cr
\hline
\text{eq.~\eqref{eq:Z32u=1}}& & .853& .133& .0083& .0044& .52\e3& .34\e3& .35\e4& .33\e4\cr
\text{\cite{CM1}}& & .918& .076& .0048& .0003& .30\e3& .25\e4& .79\e7& .19\e4\cr
\end{array}$
\end{table}

\begin{table}[htbp]
\caption{$C_3$-fields of prime conductor: Sylow 2-subgroups}
  \label{tab:syl2Z3prim}
$\begin{array}{cr|rrrrrrrr}
 D& |S|& 1& 2^2& 4^2& 2^4& 8^2& 4^2\ti2^2& 2^6& 16^2\cr
\hline
 \ge10^{22}& 6\zs& 1.001&  .997&  .999& .988&  .980& 1.005& 1.051&  .96\cr
 \ge10^{28}& 10^6& 1.000& 1.002& 1.004& .979& 1.074&  .954& 1.186& 1.26\cr
 \ge10^{32}& 10^5& 1.000& 1.003&  .992& .983&  .980&  .790&  .868&  .61\cr
\end{array}$
\end{table}

\subsection{Cyclic cubic extensions of $\QQ(\sqrt{-3})$}
As a second case we have investigated cyclic cubic extensions of the complex
quadratic number field $K_0=\QQ(\sqrt{-3})$. Here, again $u=1$, so according to
Conjecture~\ref{conj:main} a given 2-torsion $\cO$-module $H$ of 2-rank $2r$
should occur with the same probability~\eqref{eq:Z32u=1} as in the previous
case. Table~\ref{tab:syl2Z3sqrt-3} gives results on this case by listing the
quotient of the observed densities by the predicted density, for sets of
$10^5$ fields of discriminant at least $10^i$, for $i\in\{16,20,24\}$. Again,
the data seem in agreement with our conjecture.

\begin{table}[htbp]
\caption{$C_3$-fields over $\QQ(\sqrt{-3})$: Sylow 2-subgroups}
\label{tab:syl2Z3sqrt-3}
$\begin{array}{lr|rrrrrrrr}
 D& |S|& 1& 2^2& 4^2& 2^4& 8^2& 4^2\ti2^2& 2^6& 16^2\cr
\hline
 \le10^{14}& 499\,815& 1.034& .820& .817& .307& .815& .316& 0& .615\cr
\hline
 \ge10^{16}& 10^5& 1.018& .901& .951& .546& .999& .732& 0& .307\cr
 \ge10^{20}& 10^5& 1.007& .967& .965& .780& .922& .732& 1.157& .921\cr
 \ge10^{24}& 10^5& 1.002& .991& .987& .949& .922& .966& 1.157& .615\cr
\hline
\text{eq.~\eqref{eq:Z32u=1}}& & .853& .133& .0083& .0044& .52\e3& .34\e3& .35\e4& .33\e4\cr
\end{array}$
\end{table}

\subsection{Cyclic cubic extensions of $\QQ(\sqrt{5})$}
In the case of a real quadratic base field $K_0$ we have $u=2$, so
Proposition~\ref{prop:rankmoment} predicts the distribution
\begin{equation}\label{eq:Z32u=2r}
  \pr(\rnk_2(\Cl_K)=2r)=
    \frac{27}{16}\cdot\frac{(2)_\infty(16)_\infty}{(4)_\infty^2}
    \cdot\frac{1}{2^{r(r+4)}(4)_r}
\end{equation}
for the 2-ranks of class groups, with higher moments
$$\prod_{k=1}^n(1+2^{2k-5}).$$
More precisely, Conjecture~\ref{conj:main}  predicts that a 2-torsion
$\cO$-module $H$ of 2-rank $2r$ should occur with probability
\begin{equation}\label{eq:Z32u=2}
  \frac{12}{5}\cdot\frac{(2)_\infty(16)_\infty}{(4)_\infty^2}
  \cdot\frac{2^{r^2}(4)_{r+2}}{|H|^2\cdot|\Aut_\cO(H)|}
\end{equation}
as Sylow 2-subgroup of a class group of a cyclic cubic number. Data for the
case $K_0=\QQ(\sqrt{5})$ are listed in
Tables~\ref{tab:2-rnkZ3sqrt5} and~\ref{tab:syl2Z3sqrt5}.

\begin{table}[htbp]
\caption{$C_3$-fields over $\QQ(\sqrt{5})$: 2-ranks and higher moments}
\label{tab:2-rnkZ3sqrt5}
$\begin{array}{lr|rrr|rrrr}
 D& |S|& r=0& 2& 4& n=1& 2& 3& 4\cr
\hline
 \le10^{16}& 236\,832& .9672& .0327& .11\e3& 1.100& 1.518& 3.51& 16.5\cr
\hline
 \ge10^{20}&     10^5& .9627& .0370& .30\e3& 1.115& 1.631& 4.56& 30.1\cr
 \ge10^{24}&     10^5& .9596& .0401& .27\e3& 1.124& 1.670& 4.63& 28.9\cr
 \ge10^{28}&     10^5& .9594& .0402& .34\e3& 1.126& 1.690& 4.93& 33.5\cr
\hline
\span\text{eq.~\eqref{eq:Z32u=2r}}& .9597& .0400& .33\e3& 1.125& 1.687& 5.06& 45.6\cr
\span\text{CL-prediction}& .9793& .0207& .02\e3& 1.062& 1.316& 2.39& 7.7\cr
\end{array}$
\end{table}

\begin{table}[htbp]
\caption{$C_3$-fields over $\QQ(\sqrt{5})$: Sylow 2-subgroups}
\label{tab:syl2Z3sqrt5}
$\begin{array}{lrrrrrrrrr}
 D& 1& 2^2& 4^2& 2^4& 8^2& 4^2\ti2^2& 2^6\cr
\hline
 \le10^{16}& 1.008&  .816& .906& .336& 1.318& 0& 0\cr
\hline
 \ge10^{20}& 1.003&  .923& 1.073&  .918&     0&     0& 0\cr
 \ge10^{24}& 1.000& 1.003& 1.008&  .826& 2.081&     0& 0\cr
 \ge10^{28}& 1.000& 1.005& 1.089& 1.010& 2.081& 1.567& 0\cr
\hline
\text{eq.~\eqref{eq:Z32u=2}}& .960& .039& .61\e3& .33\e3& .96\e5& .64\e5& .65\e6\cr
\end{array}$
\end{table}

\subsection{Cyclic cubic extensions of $\QQ(\sqrt{-1})$}  \label{sec:6.4}
Now we choose the base field $K_0=\QQ(\sqrt{-1})$ containing the 4th roots
of unity. The relevant unit rank equals $u=1$. This situation is not covered
by the predictions made in Section~\ref{sec:conj}. Still, the data in
Table~\ref{tab:2-rnkZ3sqrt-1} seems to confirm that the 2-ranks behave
according to formula~(\ref{eq:Z32u=1r}). On the other hand, the distribution of
individual Sylow 2-subgroups shown in Table~\ref{tab:syl2Z3sqrt-1} does
{\em not} seem to follow the formulas from~(\ref{eq:Z32u=1}).


\begin{table}[htbp]
\caption{$C_3$-fields over $\QQ(\sqrt{-1})$: 2-ranks and higher moments}
\label{tab:2-rnkZ3sqrt-1}
$\begin{array}{lr|rrrr|rrrr}
 D& |S|& r=0& 2& 4& 6&  n=1& 2& 3& 4\cr
\hline
 \le10^{15}& 227\,756& .8642& .1327& .303\e2&     0& 1.444& 3.76& 21.7& 233\cr
\hline
 \ge10^{20}&     5\zf& .8555& .1403& .419\e2& .1\e4& 1.485& 4.23& 30.7& 546\cr
 \ge10^{24}&     5\zf& .8541& .1412& .470\e2& .2\e4& 1.496& 4.42& 35.4& 747\cr
 \ge10^{28}&     5\zf& .8533& .1419& .473\e2& .5\e4& 1.499& 4.52& 41.4& 1119\cr
 \ge10^{32}&     4\zf& .8527& .1425& .472\e2& .5\e4& 1.502& 4.56& 43.1& 1227\cr
\hline
\span\text{eq.~\eqref{eq:Z32u=1r}}& .8530& .1422& .474\e2& .4\e4& 1.500& 4.50& 40.5& 1336\cr
\cr
\end{array}$
\end{table}

\begin{table}[htbp]
\caption{$C_3$-fields over $\QQ(\sqrt{-1})$: Sylow 2-subgroups}
\label{tab:syl2Z3sqrt-1}
$\begin{array}{lrrrrrrrrr}
 D& 1& 2^2& 4^2& 2^4& 8^2& 4^2\ti2^2& 2^6& 16^2& 8^2\ti2^2\cr
\hline
 \le10^{15}& .864& .115& .016& .25\e2& .12\e2& .50\e3&     0& .7\e4& .6\e4\cr
\hline
 \ge10^{16}& .859& .120& .016& .30\e2& .13\e2& .42\e3&     0& .2\e3& .4\e4\cr
 \ge10^{24}& .854& .123& .017& .40\e2& .11\e2& .64\e3& .1\e4& .7\e4& .3\e4\cr
 \ge10^{32}& .853& .125& .017& .40\e2& .10\e2& .62\e3& .4\e4& .6\e4& .4\e4\cr
\hline
\text{eq.~\eqref{eq:Z32u=1}}& .853& .133& .0083& .44\e2& .52\e3& .34\e3& .35\e4& .33\e4& .21\e4\cr
\text{\cite{CM1}}& .918& .076& .0048& .03\e2& .30\e3& .25\e4& .79\e7& .19\e4& .16\e4\cr
\end{array}$
\end{table}

\section{$D_5$-extensions of $\QQ$}\label{sec:D5}

The fourth test case consists of quintic extensions of $\QQ$ with dihedral
Galois group $G=D_5$. Here, $\cO=\QQ(\sqrt{5})$ and $\chi_1$ is the sum of the
two non-rational characters of $G$ of degree~2. Again, the behaviour of the
prime $p=2$ is interesting. \par
Here, in contrast to the previous cases we don't have a fast method to
enumerate \emph{all} $D_5$-fields between given discriminant bounds nor is
there a proven asymptotic formula for the number of
such fields. Nevertheless, assuming the Cohen--Lenstra heuristic for the 5-rank
of quadratic fields, an obvious asymptotic lower bound for the number of
fields is obtained by just counting those fields whose Galois closure is
unramified over the quadratic subfield. According to this, for large $X$ there
should exist at least $0.07599\,\sqrt{X}$ complex quintic $D_5$-fields
of discriminant at most $X$, and at least $0.01507\,\sqrt{X}$ totally real
ones. \par
We have produced large sets of fields by specializing the $D_5$-polynomial
\begin{align*}
 X^5-2\,vX^4&-u(5u^2-10uv+4v^2)X^2+2u^2(5u-4v)(u-v)X\\
   &-4u^3(u-v)^2-X^2(X-u)t\ \in \QQ(u,v,t)[X]
\end{align*}
for integral $|u|,|v|\le 2500$ with $\gcd(u,v)=1$, and $|t|\le 50000$. Of
these several billion fields, in both possible signatures we retained the first
200,000 ones of discriminant at least $10^i$, where $15\le i\le 21$. A priori
there is no reason why the class groups of the fields obtained in this way
should show the same behaviour as class groups of random $D_5$-fields. Thus
our tables here should be taken with even more care than those in the previous
examples.

\subsection{Non-real $D_5$-extensions of $\QQ$}
For complex $D_5$-extensions we obtain $u=1$, so according to
Conjecture~\ref{conj:main} the (necessarily even) 2-ranks of class groups
should be distributed according to the probability in
formula~\eqref{eq:Z32u=1r}, that is to say, as in the case of cyclic cubic
fields.
\par
Despite the fact that our lists of fields are not complete, it turns out
that the distributions of 2-ranks given in Table~\ref{tab:2-rnkD5compl} is
not too far away from the prediction~\eqref{eq:Z32u=1r}.

\begin{table}[htbp]
\caption{Non-real $D_5$-fields over $\QQ$: 2-ranks and higher moments}
\label{tab:2-rnkD5compl}
$\begin{array}{lr|rrrr|llr}
 D& |S|& r=0& 2& 4& 6&  n=1& 2& 3\cr
\hline
 \le10^{15}& 1\,183\,056& .9266& .0724& .097\e2& .8\e6& 1.232& 2.34& 9.7\cr
\hline
 \ge10^{17}& 200\,000& .9110& .0869& .216\e2& .05\e4& 1.293& 2.87& 16.6\cr
 \ge10^{19}& 200\,000& .9049& .0922& .290\e2& .55\e4& 1.324& 3.35& 33.1\cr
 \ge10^{21}& 200\,000& .8992& .0972& .354\e2& .20\e4& 1.346& 3.46& 28.2\cr
\hline
\span\text{eq.~\eqref{eq:Z32u=1r}}& .8530& .1422& .474\e2& .38\e4& 1.5& 4.5& 40.5\cr
\end{array}$
\end{table}

\subsection{Totally real $D_5$-extensions of $\QQ$}
For totally real $D_5$-extensions we obtain $u=2$, so according to
Conjecture~\ref{conj:main} the 2-ranks of class groups should be distributed
according to formula~\eqref{eq:Z32u=2r}.

\begin{table}[htbp]
\caption{Totally real $D_5$-fields over $\QQ$: 2-ranks and higher moments}
\label{tab:2-rnkD5real}
$\begin{array}{lr|rrrr|rrr}
 D& |S|& r=0& 2& 4& 6&  n=1& 2& 3\cr
\hline
 \le10^{14}& 147\,683& .9876& .0124& .14\e4& 0& 1.037& 1.19& 1.84\cr
\hline
 \ge10^{17}& 200\,000& .9789& .0210& .85\e4& 0& 1.064& 1.34& 2.67\cr
 \ge10^{21}& 200\,000& .9721& .0276& .19\e3& 0& 1.086& 1.46& 3.54\cr
\hline
\span\text{eq.~\eqref{eq:Z32u=2r}}& .9597& .0400& .33\e3& .66\e6& 1.125& 1.69& 5.06\cr
\end{array}$
\end{table}

\goodbreak

\end{document}